\documentclass{amsart}

\usepackage{amsmath,amssymb,amsthm}
\usepackage{pinlabel}

\newtheorem{prop}{Proposition}
\newtheorem{thm}[prop]{Theorem}
\newtheorem{lem}[prop]{Lemma}
\newtheorem{cor}[prop]{Corollary}

\theoremstyle{definition}

\newtheorem{rem}[prop]{Remark}

\newtheorem*{ack}{Acknowledgements}

\def\co{\colon\thinspace}

\newcommand{\C}{\mathbb C}
\newcommand{\CP}{\mathbb{C}\mathrm{P}}

\newcommand{\rme}{\mathrm e}

\newcommand{\rmi}{\mathrm i}

\newcommand{\MM}{\mathcal M}
\newcommand{\wtMM}{\widetilde{\mathcal M}}

\newcommand{\R}{\mathbb R}
\newcommand{\RP}{\mathbb{R}\mathrm{P}}

\newcommand{\wtW}{\widetilde{W}}

\newcommand{\oz}{\overline{z}}
\newcommand{\Z}{\mathbb Z}

\newcommand{\lra}{\longrightarrow}

\newcommand{\ip}{\,\rule{2.3mm}{.2mm}\rule{.2mm}{2.3mm}\; }

\DeclareMathOperator{\Aut}{Aut}

\DeclareMathOperator{\ev}{ev}

\DeclareMathOperator{\FS}{FS}

\DeclareMathOperator{\Int}{Int}

\begin{document}

\author[M.~D\"orner]{Max D\"orner}
\author[H.~Geiges]{Hansj\"org Geiges}
\author[K.~Zehmisch]{Kai Zehmisch}
\address{Mathematisches Institut, Universit\"at zu K\"oln,
Weyertal 86--90, 50931 K\"oln, Germany}
\email{mdoerner@math.uni-koeln.de}
\email{geiges@math.uni-koeln.de}
\email{kai.zehmisch@math.uni-koeln.de}

\title{Open books and the Weinstein conjecture}

\date{}

\begin{abstract}
We show the existence of a contractible periodic Reeb orbit
for any contact structure supported by an open book whose binding
can be realised as a hypersurface of restricted contact type in a
subcritical Stein manifold. A key ingredient in the proof
is a higher-dimensional version of Eliashberg's theorem
about symplectic cobordisms from a contact manifold to
a symplectic fibration.
\end{abstract}



\maketitle


\section{Introduction}
A contact manifold $(M,\xi)$ is said to satisfy the
Weinstein conjecture~\cite{wein79} if for every contact form $\alpha$
defining $\xi=\ker\alpha$ the corresponding Reeb vector field
$R_{\alpha}$ has a closed orbit, cf.~\cite{pasq12}. In a seminal paper,
Hofer~\cite{hofe93} developed a method to approach this
conjecture in dimension~$3$ via the study of holomorphic discs in the
symplectisation of $(M,\xi)$. This allowed him to prove the conjecture
for overtwisted contact $3$-manifolds and contact structures on
$3$-manifolds with non-vanishing second homotopy group.
In \cite{geze} we explored the ramifications of Hofer's method
in dimension~$3$ (including applications to symplectic $4$-manifolds).
In~\cite{geze12} we extended this approach to higher dimensions.
Using a capping construction that can be traced back to
McDuff~\cite{mcdu91} we were able to work with holomorphic spheres
rather than discs, which simplifies the holomorphic analysis.

In dimension~$3$ the Weinstein conjecture has now been answered
in the affirmative by Taubes~\cite{taub07}, cf.~\cite{hutc10},
with the help of Seiberg--Witten theory. Prior to the work of Taubes, 
arguably the most important
advance on the $3$-dimensional Weinstein conjecture had been due
to Abbas--Cieliebak--Hofer~\cite{ach05}, who established the
conjecture for all contact structures supported
(in the sense of Giroux~\cite{giro02})
by a planar open book. 

This paper continues the programme begun in~\cite{geze12}
and is close in spirit to the work of Abbas et al.,
in the sense that we establish the Weinstein conjecture
for contact structures (in dimension $\geq 5$) supported
by suitable open books. Specifically,
the property we require is that the binding of the open book
be realisable as a hypersurface of restricted contact type
in a subcritical Stein manifold,
see Theorem~\ref{thm:main}. (Recall that a Stein manifold
of real dimension $2n$ has the homotopy type of an $n$-dimensional complex.
It is called subcritical if its handle decomposition
only contains handles up to dimension~$n-1$.)
According to Giroux, any compact contact manifold $(M^{2n+1},\xi)$
admits a supporting open book whose pages are Stein manifolds.
Thus, the assumptions of our existence theorem for
periodic Reeb orbits will in particular be
satisfied if the pages are subcritical.
\section{The main theorem}
\label{section:main}

Let $M$ be a closed, oriented manifold of dimension $2n+1\geq 5$
admitting a (cooriented) contact structure $\xi=\ker\alpha$, i.e.\
$\alpha\wedge (d\alpha)^n>0$. The Reeb vector field $R$ of the contact
form $\alpha$ is given by $R\ip d\alpha\equiv 0$ and $\alpha(R)\equiv 1$.

Recall that an open book
decomposition of $M$ is a pair $(B,\theta)$ consisting
of an oriented codimension~$2$ submanifold $B\subset M$ with trivial
normal bundle and a locally trivial fibration
$\theta\co M\setminus B\rightarrow S^1$ given in a neighbourhood
$B\times D^2\subset M$ of $B$ by the angular coordinate in
the $D^2$-factor. The manifold $B$ is called the \emph{binding};
the closures $\Sigma$ of the fibres of $\theta$
are referred to as the \emph{pages} of the open book. The pages
are oriented consistently with the orientation of their boundary~$B$.

The contact structure $\xi$ is said to be \emph{supported}~\cite{giro02}
by the open book decomposition $(B,\theta)$ if a contact form $\alpha$
can be chosen such that
\begin{itemize}
\item[(ob-i)] the $2$-form $d\alpha$ induces a positive symplectic form on
each fibre of~$\theta$, and
\item[(ob-ii)] the $1$-form $\alpha$ induces a positive contact form on~$B$.
\end{itemize}

\begin{thm}
\label{thm:main}
Let $(M,\xi)$ be a contact manifold of dimension $\geq 5$
with a supporting open book $(B,\theta)$ whose
binding $B$ (with the induced contact structure)
embeds as a hypersurface of restricted contact type into
a subcritical Stein manifold. Then any contact form defining $\xi$ has
a contractible periodic Reeb orbit.
\end{thm}

Note that the theorem claims the existence of a closed Reeb orbit for
\emph{any} contact form defining~$\xi$, not only those
adapted to the open book in the sense of (ob-i) and (ob-ii).
\section{A neighbourhood theorem}
\label{section:nbhd-B}
Our aim in this section is to prove a neighbourhood theorem for
the binding $B$ of an open book.
So we consider contact structures on $B\times D^2$ that satisfy
the compatibility conditions (ob-i) and (ob-ii) with the open book structure
on this neighbourhood of the binding.

For the purposes of this section we call
\[ \Sigma_{\theta}:=\{ (b;r\rme^{\rmi\theta})\in B\times D^2\co
b\in B,\, 0\leq r\leq 1\}\]
the `page', and we identify
$B$ with $B\times\{0\}\subset B\times D^2$.

\begin{prop}
\label{prop:normal}
Let $\alpha$ be a contact form on $B\times D^2$ with the following
properties:
\begin{itemize}
\item[(i)] The $1$-form $\alpha_B:=\alpha|_{TB}$ is a contact form on~$B$.
\item[(ii)] For each $\theta\in [0,2\pi)$, the $2$-form
$d\alpha|_{T\Sigma_{\theta}}$ is a symplectic form on $\Sigma_{\theta}
\setminus B$.
\item[(iii)] With the orientations of $B$ and $\Sigma_{\theta}$ induced by
$\alpha_B$ and $d\alpha|_{T\Sigma_{\theta}}$, respectively, $B$ is oriented
as the boundary of $\Sigma_{\theta}$.
\end{itemize}
Then for $\varepsilon >0$ sufficiently small
there is a page-preserving embedding
\[ B\times D^2_{\varepsilon}\longrightarrow B\times D^2\]
which is the identity on $B\times\{ 0\}$ and pulls back $\alpha$ ---
after an isotopic modification of $\alpha$ near $B\times\{0\}$
via forms satisfying {\rm (i)} (up to constant scale)
to {\rm (iii)} --- to a $1$-form
\[ h_1(r)\, \alpha_B+h_2(r)\, d\theta\]
satisfying
\begin{itemize}
\item $h_1(0)>0$, $h_2(r)=r^2$ near $r=0$,
\item $h_1^{n-1}(h_1h_2'-h_1'h_2)>0$ (the contact condition),
\item $h_1'<0$ for $r>0$ (the symplectic condition on the pages).
\end{itemize}
\end{prop}

\begin{proof}
The idea of the proof (due to Giroux) is to identify a neighbourhood
of the binding with $B\times D^2_{\varepsilon}$ in such a way that the
radial direction corresponds to the (negative) characteristic direction
in the pages.

Choose volume forms $\Omega_{\theta}$ on the pages $\Sigma_{\theta}$
compatible with the orientation of the pages, e.g.\
\[ \Omega_{\theta}=-(\cos\theta\, dx+\sin\theta\, dy)\wedge
\alpha_B\wedge (d\alpha_B)^{n-1}.\]
Write $\beta_{\theta}:=\alpha|_{T\Sigma_{\theta}}$ for the
restriction of the contact form $\alpha$ to the page $\Sigma_{\theta}$.
The characteristic vector field $X_{\theta}$ is the vector
field tangent to $\Sigma_{\theta}$ defined by
\[ X_{\theta}\ip\Omega_{\theta}=
-\beta_{\theta}\wedge (d\beta_{\theta})^{n-1},\]
see \cite{giro91} and \cite[Section~2.5.4]{geig08}.
By conditions (i) and (iii), this vector field $X_{\theta}$ is transverse
to $B\subset\Sigma_{\theta}$ and points into the page.

Along the binding we have a whole $S^1$-family of vector fields,
so it is not immediately clear that the flow of the $X_{\theta}$
defines a smooth embedding. To circumvent this problem, we
replace the time~$r$ flow of the $X_{\theta}$ by the time~$1$ flow
of a smooth $D^2$-family of `honest' vector fields on $B\times D^2$.
This was suggested to us by Fan Ding.

Write the points in the parameter space $D^2$ in polar
coordinates as $R\rme^{\rmi\varphi}$.
Define a smooth $D^2$-family of smooth
vector fields on $B\times D^2$ by
\[ \tilde{X}_{R\rme^{\rmi\varphi}}\ip\bigl( dx\wedge dy\wedge
\alpha_B\wedge (d\alpha_B)^{n-1}\bigr) =
(R\sin\varphi\; dx-R\cos\varphi\; dy)\wedge\alpha\wedge (d\alpha)^{n-1}.\]
We claim that
\[ \tilde{X}_{R\rme^{\rmi\theta}}|_{\Sigma_{\theta}}=RX_{\theta}.\]
Indeed, by taking the wedge product of the defining equation for $X_{\theta}$
with $\sin\theta\, dx-\cos\theta\, dy$ and observing that
\[ T\Sigma_{\theta}\subset\ker (\sin\theta\, dx-\cos\theta\, dy), \]
we have
\[ X_{\theta}\ip \bigl( dx\wedge dy
\wedge
\alpha_B\wedge (d\alpha_B)^{n-1}\bigr)= 
(\sin\theta\, dx-\cos\theta\, dy)\wedge \alpha\wedge
(d\alpha)^{n-1},\]
from which the claim follows.

Write $\tilde{\psi}_t^{R,\varphi}$ for the time~$t$
flow of $\tilde{X}_{R\rme^{\rmi\varphi}}$, and $\psi_t^{\theta}$
for the time~$t$ flow of $X_{\theta}$ (along $\Sigma_{\theta}$).
For $\varepsilon>0$ sufficiently small,
this defines an embedding
\[ \begin{array}{rccc}
\Psi\co & B\times D^2_{\varepsilon} & \longrightarrow &
                           B\times D^2\\
        & (b,r\rme^{\rmi\theta})                      & \longmapsto     &
                           \tilde{\psi}_1^{r,\theta}(b,0),
\end{array}\]
where $D^2_{\varepsilon}$ denotes the $2$-disc of radius~$\varepsilon$.

We maintain that
\[ \ker\bigl(\Psi^*\alpha\bigr)=\ker (\alpha_B+h\, d\theta)\]
for some smooth function $h$ on $B\times D^2_{\varepsilon}$
that makes the $1$-form $h\, d\theta$ well defined (in particular,
$h=0$ along~$B$) and smooth.
Indeed, a straightforward computation, cf.~\cite[Lemma~2.5.20]{geig08},
gives
\[ X_{\theta}\in\ker\beta_{\theta}\;\;\;\text{and}\;\;\;
L_{X_{\theta}}\beta_{\theta}=\mu_{\theta}\beta_{\theta}\]
for some smooth function $\mu_{\theta}\co\Sigma_{\theta}\rightarrow\R$.
Since
\[ \tilde{\psi}_1^{r,\theta}(b,0) = \psi_r^{\theta}(b,0),\]
we have $T\Psi(\partial_r|_{\Sigma_{\theta}})=X_{\theta}$.
So the first condition on $X_{\theta}$ means that
$\Psi^*\alpha$ contains no $dr$-component; the second condition
means that $\ker\bigl(\Psi^*\alpha|_{T\Sigma_{\theta}}\bigr)$
is invariant under the flow of $\partial_r$. This proves the claim.

The standard Gray stability argument applied to the linear
interpolation of the contact forms $\alpha_B+h\, d\theta$
and $\alpha_B+r^2\, d\theta$ gives a flow fixed at $B$ and tangent to
the $\partial_r$-direction which isotopes the contact structure
$\ker (\alpha_B+h\, d\theta)$ to $\ker(\alpha_B+r^2\, d\theta)$.

In sum, we have an embedding $B\times D^2_{\varepsilon}
\rightarrow B\times D^2$ (possibly
for some smaller $\varepsilon >0$) that pulls back
$\alpha$ to a $1$-form
$f(\alpha_B+r^2\, d\theta)$ with a positive function
$f$ satisfying $f\equiv 1$ along $B\times\{0\}$. Observe that
conditions (ii) and (iii) are equivalent to $\partial f/\partial r<0$.
For ease of notation, we continue to write $\alpha$ for this
pulled-back form.

Now let $\lambda\co [0,\varepsilon]\rightarrow[0,1]$ be
a smooth function that is identically $1$ near $r=0$ and identically $0$
near $r=\varepsilon$. Replace $\alpha$ by
\[ \bigl(\lambda(r)(c-r^2)+(1-\lambda(r))f\bigr)
\, (\alpha_B+r^2\, d\theta),\]
where $c$ is a positive constant. This new $1$-form coincides
with $\alpha$ away from $B\times\{0\}$, and is of the
desired form $h_1(r)\,\alpha_B+h_2(r)\, d\theta$
near $B\times\{0\}$. For $c$ sufficiently large, the
coefficient function $\lambda(r)(c-r^2)+(1-\lambda(r))f$
has negative $r$-derivative, so condition (ii)
will be satisfied. The linear interpolation between $\alpha$ and this
modified form is via forms satisfying (i) to (iii).
\end{proof}

\begin{rem}
\label{rem:h}
For the capping construction below it is convenient to make
a specific choice of the functions $h_1,h_2$ subject to
the following conditions:
\begin{itemize}
\item[(i)] $h_1(r)=c\rme^{-r^2}$ near $r=0$ for some constant~$c$,
\item[(ii)] $h_1(r)=c\rme^{-r}$ and $h_2(r)$ constant for $r$ near some
$r_0>0$.
\end{itemize}
Figure~\ref{figure:interpolate} illustrates that for $c$ sufficiently
large it is possible to find such a pair $(h_1,h_2)$ that can be
interpolated to the given pair (drawn as a dashed line) via
pairs satisfying the conditions stipulated in the preceding lemma.
\end{rem}

\begin{figure}[h]
\labellist
\small\hair 2pt
\pinlabel $h_1$ [t] at 384 37
\pinlabel $h_2$ [r] at 37 276
\endlabellist
\centering
\includegraphics[scale=0.4]{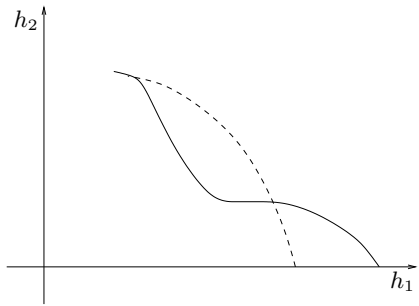}
  \caption{How to choose a specific $(h_1,h_2)$.}
  \label{figure:interpolate}
\end{figure}

\begin{rem}
An open book can alternatively be described by a
compact $2n$-dimen\-sional manifold $\Sigma$ with non-empty
boundary $\partial\Sigma=:B$, and a monodromy diffeomorphism
$\phi$ of $\Sigma$, equal to the identity near~$\partial\Sigma$.
The closed manifold $M$ obtained from the mapping cylinder
of $(\Sigma,\phi)$ (which has boundary
$\partial\Sigma\times S^1$) by attaching $\partial\Sigma\times D^2$
then has an open book decomposition $(B,\theta)$.
Conversely, from $(B,\theta)$ one can recover $(\Sigma,\phi)$,
cf.~\cite[Section~4.4.2]{geig08}.

A construction of Giroux~\cite{giro02},
cf.~\cite[Section~7.3]{geig08}, produces a contact
structure adapted to an open book, provided the page is exact symplectic
with contact type boundary, and the monodromy is symplectic.
Here the contact structure has, by construction, the normal form
described in Proposition~\ref{prop:normal}.
So a consequence of the proposition is that any
contact structure supported by an open book can be obtained
via Giroux's construction.
\end{rem}
\section{From open books to symplectic fibrations}
The following theorem and its corollary are higher-dimensional analogues
of \cite[Theorem~1.1]{elia04}, which was the crucial
part in Eliashberg's construction of symplectic caps for
weak symplectic fillings. Our proof is close in spirit
to that of Eliashberg.

In the theorem and throughout this section, $(M,\xi=\ker\alpha)$
will be a closed contact manifold and $(B,\theta)$ an open book
decomposition of $M$ to which the contact form $\alpha$ is adapted.
We write $\Sigma$ for the page of the open book
decomposition, and $\alpha_B$ for the restriction of $\alpha$ to~$TB$.

\begin{thm}
\label{thm:cobordism}
Suppose there is a (not necessarily compact) symplectic manifold
$(C,\omega_C)$ with strong concave boundary that is strictly
contactomorphic to $(B,\alpha_B)$.

Then there is a (not necessarily compact) symplectic manifold $(W,\omega)$
with boundary $\partial W=-M\sqcup N$, where $(M,\alpha)$ is a
strong concave boundary component, and $N$ is a fibre
bundle over $S^1$ with fibre $F=\Sigma\cup C$ such that
$\omega$ restricts to a symplectic form on each fibre. Moreover,
the holonomy of the fibre bundle $N\rightarrow S^1$ (given by the
kernel of $\omega|_{TN}$) is the identity on the subset $C\subset F$
of the fibre.
\end{thm}

The condition on $C$ means that there is a Liouville
vector field $Y$ defined near the boundary $\partial C$ and
pointing inwards, such that $(-\partial C,i_Y\omega_C|_{T(\partial C)})$
is diffeomorphic to $(B,\alpha_B)$.

\begin{cor}
Given any closed contact manifold $(M,\xi)$, there is
a compact symplectic manifold $(W,\omega)$ with boundary
$\partial W=-M\sqcup N$ such that $(M,\xi)$ is a strong concave
boundary component and $(N,\omega|_{TN})$ is a 
symplectic fibre bundle over~$S^1$.
\end{cor}

\begin{proof}
By \cite{giro02} we may assume that $(M,\xi)$ is supported
by an open book whose page $\Sigma$ is a compact Stein manifold, i.e.\
a compact sublevel set of an exhausting, strictly plurisubharmonic
function on a Stein domain. According
to \cite[Theorem~3.1]{lima97}, this Stein manifold $\Sigma$
embeds into a smooth projective variety. This means, in particular,
that we find a \emph{compact} symplectic cap $(C,\omega_C)$ for
$\Sigma$ as required by Theorem~\ref{thm:cobordism}.
\end{proof}

\begin{proof}[Proof of Theorem~\ref{thm:cobordism}]
Topologically, the definition of $W$ is very simple. Let
$B\times D^2\subset M$ be a neighbourhood of the binding
as in Section~\ref{section:nbhd-B}. Then define
\[ W:= [-2,0]\times M\cup_{B\times D^2} (C\times D^2)\]
(and smooth corners),
where we think of $B\times D^2$ as, on the one hand, a subset
of $M\times\{0\}$ and, on the other, as $\partial C\times D^2$.
(We take a cylinder of height $2$ over $M$ simply to have enough
room for the construction that follows.)

Symplectically, we want to think of $[-2,0]\times M$
as a part of the symplectisation of~$M$, i.e.\ we equip it
with the symplectic form $d(\rme^s\alpha)$. On $C\times D^2$
we take the symplectic form $\omega_C+f'(\rho)\,d\rho\wedge d\theta$,
where $(\rho,\theta)$ denote polar coordinates on~$D^2$, and
$f\co \R_0^+\rightarrow\R_0^+$ is a smooth function with the
following properties:
\begin{itemize}
\item[(f-i)] $f(\rho)=\rho^2$ near $\rho=0$ (so that the form above is smooth),
\item[(f-ii)] $f'(\rho)>0$ for $\rho>0$ (so that the form is symplectic),
\item[(f-iii)] $f(\rho)=\rho$ near $\rho=1$ (this will be relevant later,
see condition (J1) in Section~\ref{section:proof}).
\end{itemize}

For the gluing, we work in the neighbourhood
\[ U:= [-1,0]\times B\times D^2\subset [-2,0]\times M,\]
with polar coordinates $(r,\theta)$ on the $D^2$-factor.
Purely for notational convenience, we shall assume that the contact form
$\alpha$ is given on an open neighbourhood of $B\times D^2$ in $M$
by $h_1(r)\, \alpha_B+h_2(r)\, d\theta$ with
\begin{itemize}
\item[(h-i)] $h_1(r)=\rme^{-r^2}$ and $h_2(r)=r^2$ near $r=0$,
\item[(h-ii)] $h_1(r)=\rme^{-r}$ and $h_2(r)\equiv 1$ near $r=1$.
\end{itemize}
By Remark~\ref{rem:h}, this can always be achieved
(up to some inessential constants).

Both the gluing along parts of the boundary and the
smoothing of corners are awkward processes in the presence of
a symplectic (or any other geometric) structure. To avoid this,
the idea is to construct a model which contains a
symplectic copy both of $C\times D^2$ and $U$, and such that
the identification of $[-2,0]\times M$ and this model along~$U$
(or in fact a small open neighbourhood of it) realises the
desired topological construction.

By the assumptions of the theorem, we obtain a symplectic manifold
by gluing $(-\infty,0]\times B$ and $C$ along $B\equiv
\{0\}\times B\equiv-\partial C$. The model in question is
the product of this manifold with a $D^2$-factor:
\[ (W_0,\omega_0):=\Bigl(\bigl( (-\infty,0]\times B,d(\rme^t\alpha_B)\bigr)
\cup_B
(C,\omega_C)\Bigr)\times \bigl(D^2,f'(\rho)\, d\rho\wedge d\theta\bigr),\]
see Figure~\ref{figure:model}. In that figure, the left part
of the horizontal axis represents $(-\infty,0]\times B$;
the right part, $C$. The vertical axis gives the $\rho$-coordinate
of the $D^2$-factor, so a `realistic' picture is given by
rotating the figure around the horizontal axis.

\begin{figure}[h]
\labellist
\small\hair 2pt
\pinlabel $\rho$ [r] at 352 184
\pinlabel $t=0$ [t] at 361 0
\pinlabel $t=-1$ [t] at 73 0
\pinlabel $C$ [t] at 485 0
\pinlabel ${C\times D^2}$ at 485 80
\pinlabel $\Gamma$ [bl] at 287 108
\pinlabel ${B\times D^2}$ [bl] at 402 186
\pinlabel $\Phi(U)$ [r] at 166 50
\endlabellist
\centering
\includegraphics[scale=0.5]{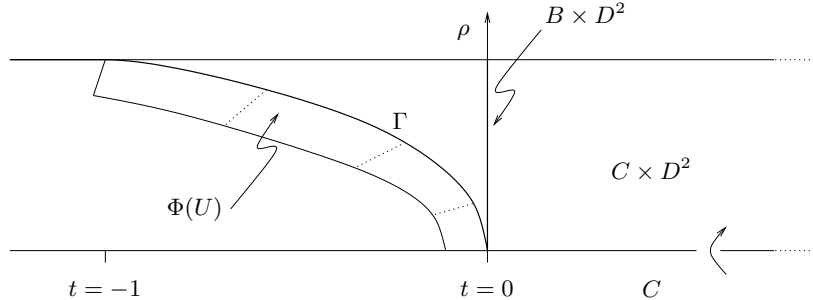} 
  \caption{The model $(W_0,\omega_0)$ for the symplectic gluing.}
  \label{figure:model}
\end{figure}

We claim that we can find a symplectomorphic copy $\Phi(U)$
of $U$ in this model as indicated in the figure. The dotted lines
represent flow lines of the Liouville vector field
$T\Phi(\partial_s)$. The hypersurface $\Gamma$ in the model
is then a strictly contactomorphic copy of $(B\times D^2,\alpha)$.

The trick, as it were, is to think of a neighbourhood
of $B\times D^2$ in $[-2,0]\times M$ not as
a neighbourhood to the left of the horizontal axis in the model
(in which case the gluing would produce
a corner at $\{0\}\times B\times\{\rho=1\}$),
but as a neighbourhood below the `curve'~$\Gamma$ (which
connects smoothly with the curve $\{\rho=1\}$ at $t=-1$).
We can then glue the part of the model to the right of $\Gamma$
in a smooth and symplectic fashion to $[-2,0]\times M$.

Write
\[ \lambda:=\rme^t\alpha_B+f(\rho)\, d\theta\]
for the primitive of $\omega_0$ on $(-\infty,0]\times B\times D^2\subset W_0$.
The corresponding Liouville vector field is
\[ Y=\partial_t+\frac{f(\rho)}{f'(\rho)}\,\partial_{\rho}.\]
Now define
\[ \Phi\co U=[-1,0]\times B\times D^2\longrightarrow
(-\infty,0]\times B\times D^2\subset W_0\]
by
\[ (s,b,r,\theta)\longmapsto (s+\log(h_1(r)),b,f^{-1}(\rme^s h_2(r)),\theta).\]

\begin{lem}
The map $\Phi$ is a symplectic embedding with $\Phi^*\lambda=\rme^s\alpha$
(and hence $T\Phi(\partial_s)=Y$).
\end{lem}

\begin{proof}
Near $r=0$ we have $f^{-1}(\rme^s h_2(r))=\rme^{s/2}r$ by (f-i) and
(h-i), so $\Phi$ is smooth.

In order to see that $\Phi$ is injective, assume that we
have
\[ \Phi(s_1,b,r_1,\theta)=\Phi(s_2,b,r_2,\theta).\]
By looking at the first and third component of the image, we see that
\[ \rme^{s_1}h_1(r_1)=\rme^{s_2}h_1(r_2)\;\;\;\text{and}\;\;\;
\rme^{s_1}h_2(r_1)=\rme^{s_2}h_2(r_2),\]
whence
\[ \frac{h_2}{h_1}(r_1)=\frac{h_2}{h_1}(r_2).\]
By the contact condition, we have $(h_2/h_1)'<0$. So this
last equation forces $r_1=r_2$ and hence $s_1=s_2$.

The fact that $\Phi^*\lambda=\rme^s\alpha$ follows from a quick calculation.
In particular, this makes $\Phi$ symplectic, and hence an immersion.
\end{proof}

The submanifold $\Gamma$ in Figure~\ref{figure:model} is defined to be
\[ \Gamma:= \Phi (\{0\}\times B\times D^2).\]
By construction, this is a hypersurface of contact type (transverse
to the Liouville vector field~$Y$), and the induced contact
form $\lambda|_{T\Gamma}$ pulls back to~$\alpha$ under the
embedding~$\Phi$.

\begin{lem}
The hypersurface $\Gamma\subset (-\infty,0]\times B\times D^2$
can be described as a graph
\[ \Gamma=\{ (t,b,\rho,\theta)\in
(-\infty,0]\times B\times D^2\co t\in [-1,0],\; \rho=g(t)\},\]
where
\[ g(t):= f^{-1}(h_2(h_1^{-1}(\rme^t))).\]
\end{lem}

\begin{proof}
The hypersurface $\Gamma$ is given by the $(t,b,\rho,\theta)$ with
\[ t=\log h_1(r)\;\;\;\text{and}\;\;\; \rho=f^{-1}(h_2(r)).\]
This clearly translates into the form in the lemma.
\end{proof}

Note that $g'(t)\leq 0$. Near $t=0$ we have
$g(t)=\sqrt{-t}$ by (f-i) and (h-i). This shows that $\Gamma$
looks like a `paraboloid' near $t=0$. In particular, this
verifies again the smoothness of~$\Gamma$.

Near $t=-1$ we have $g\equiv 1$ by
(h-ii) and (f-iii). This means that $\Gamma$ coincides with
$(-\infty,0]\times B\times\partial D^2$ near $t=-1$. Therefore,
the part of $W_0$ to the right of $\Gamma$, which is
diffeomorphic to $C\times D^2$, can be glued to
$[-2,0]\times M$ along $\{0\}\times B\times D^2$, resulting
in a symplectic manifold $(W,\omega)$
with (smooth) boundary the disjoint union of
$\{-2\}\times M$ and
\[ N:= \Bigl( M\setminus (B\times \Int(D^2))\Bigr) \cup_{B\times S^1}
\Bigl( \bigl( ([-1,0]\times B)\cup_B C\bigr)\times S^1\Bigr).\]
This manifold $N$ fibres in an obvious way over $S^1$, with fibre
given by
\[ (\Sigma\setminus\{r<1\}) \cup ([-1,0]\times B) \cup C,\]
which topologically is just $\Sigma\cup C$. The restriction of $\omega$
to this fibre is given by $d\alpha$ on the first part,
by $d(\rme^t\alpha_B)$ on the second, and by $\omega_C$ on the third.
So the fibre is indeed symplectic. Finally, the symplectic monodromy
of the $S^1$-fibration
\[ (C\times S^1,\omega_C)\longrightarrow S^1\]
is obviously the identity. This completes the proof
of Theorem~\ref{thm:cobordism}.
\end{proof}

\begin{rem}
After this result was presented by one of us (M.~D.)\ in our local
\emph{Arbeitsgemeinschaft}, an alternative proof was found
by M.~Klukas~\cite{kluk12}.
\end{rem}

The open book fibration $\theta\co M\setminus B\rightarrow S^1$
and the projection onto the angular coordinate $\theta$
in the $D^2$-factor of the model $W_0$ (outside $0\in D^2$)
induce a fibration over $S^1$ of a whole collar neighbourhood
of $N$ in $W$. We continue to write $\theta$ for this
fibration and $d\theta$ for the induced $1$-form on this collar.

Recall that a \emph{stable Hamiltonian structure} on a
manifold $N$ of dimension $2n+1$ is a pair $(\Omega,\gamma)$
consisting of a $2$-form $\Omega$ and a $1$-form
$\gamma$ such that $\gamma\wedge\Omega^n\neq 0$ and
$\ker\Omega\subset\ker d\gamma$, see~\cite{cimo05}.
The \emph{Reeb vector field} $R$ of such a stable Hamiltonian structure
is defined by $R\in\ker\Omega$ and $\gamma(R)=1$.
The \emph{symplectisation} of $(\Omega,\gamma)$ is the symplectic manifold
$\bigl((-\varepsilon,\varepsilon)\times N,\Omega+d(\rho\gamma)\bigr)$
with $\varepsilon >0$ sufficiently small.

With $N$ the manifold we just constructed, and
$\omega_N:=\omega|_{TN}$, the pair $(\omega_N,d\theta)$
is an example of such a stable Hamiltonian structure, as
is clear from the symplectic fibration property of~$N$.
The following lemma says that the symplectic form $\omega$
looks like the symplectisation of $(\omega_N,d\theta)$
in a collar neighbourhood of $N$ in~$W$.

\begin{lem}
\label{lem:collar}
There is a collar neighbourhood $(-\varepsilon,0]\times N$
of $N$ in $W$ on which the symplectic form $\omega$ can be written as
\[ \omega = \omega_N+d\rho\wedge d\theta,\]
where, by slight abuse of notation, $\rho$ denotes the collar
parameter.
\end{lem}

\begin{proof}
By condition (f-iii), the symplectic form looks as claimed
near the part of $N$ made up of $[-1,0]\times B\times S^1$ and 
$C\times S^1$, with $\rho$ the radial parameter
in the $D^2$-factor. Up to an additive constant, we may take
this $\rho$ as the collar parameter. In that region
$\omega_N$ equals $d(e^t\alpha_B)$ and $\omega_C$, respectively.

Define a non-vanishing vector field $X$ on a neighbourhood
of $N\subset W$ by
\[ X\ip\omega=d\theta.\]
Then $L_X\omega=0$,
so the flow of $X$ leaves $\omega$ invariant. Near the part
of $N$ considered previously we have $X=\partial_{\rho}$.

We claim that $X$ is transverse to $N$ everywhere. By what we have
just observed, this is only an issue near the part 
$M\setminus(B\times D^2)$ of~$N$, where $\omega_N=d\alpha$.
By condition (ob-i),
together with the orientation conditions, the Reeb vector field
$R_{\alpha}$ of $\alpha$ (which coincides up to scale
with the Reeb vector field $R$ of the stable Hamiltonian structure)
is positively transverse to the fibres of the
open book, hence
\[ \omega (X,R)=d\theta(R)>0.\]
On the other hand, if $X$ were tangent to $N\subset\partial W$
at some point of $M\setminus(B\times D^2)\subset N$, then at that point
we would have
\[ \omega (X,R)=\omega_N(X,R)=d\alpha(X,R)=0,\]
so that cannot happen.

Now define the collar parameter $\rho$ by the flow of $X$.
Then, by the defining equation for $X$ and the invariance of $\omega$,
the description of $\omega$ is as claimed in the lemma.
\end{proof}

\section{Proof of Theorem~\ref{thm:main}}
\label{section:proof}
The proof of Theorem~\ref{thm:main} is based on the study of
holomorphic curves in a symplectic cobordism having
$(M,\xi)$ as its concave end. For the construction of this cobordism
we rely on Theorem~\ref{thm:cobordism}, where we made a specific choice
of contact form~$\alpha$. Nonetheless, our main theorem
holds for \emph{any} contact form defining~$\xi$, since up to
a constant factor any such form can be realised on the concave end
of a suitable symplectic cobordism by adding a negative
half-symplectisation to the cobordism we are about to construct
and choosing a suitable graph-like hypersurface in that
cylindrical end, cf.~\cite[Remark~3.10]{geze} or \cite[p.~265]{geze12}.

At some point in the proof of Theorem~\ref{thm:main} we appeal
to a compactness theorem from symplectic field theory, which
presupposes the contact form defining $\xi$ to be non-degenerate.
Again, this is not a restriction on the allowable contact forms,
since the degenerate case follows from the non-degenerate one
by an Arzel\`a--Ascoli type argument as in~\cite{ach05},
cf.~\cite[Remark~6.4]{geze} or \cite[Section~6.4]{geze12}.
\subsection{Construction of the cobordism}
\label{section:cobordism}
By assumption, the binding $(B,\xi\cap TB)$ embeds into a subcritical Stein
manifold as a hypersurface of restricted contact type. According to a
result of Cieliebak~\cite{ciel02}, cf.~\cite[Theorem~14.16]{ciel12}, any such
subcritical Stein manifold is symplectomorphic to a split one
$(V\times\C,J_V\oplus i)$, so we may think of $B$ as a hypersurface in
$V\times\C$. On that split Stein manifold we have a strictly
plurisubharmonic function $\psi=\psi_V+|z|^2/4$, with $\psi_V$
strictly plurisubharmonic on the Stein manifold~$V$. Choose a regular
level set $S:=\psi^{-1}(c)$ of $\psi$ such that $B$ is contained in the
interior of the corresponding sublevel set. Write $A_{BS}$ for the
resulting cobordism between $B$ and~$S$.

The condition on $B$ being of \emph{restricted} contact type tells us that
there is a global Liouville vector field on the symplectic manifold
$V\times\C$ transverse to $B$ and inducing a contact form
for $\xi\cap TB$ there.
(The Liouville vector field necessarily points into $A_{BS}$
along~$B$, see the proof of \cite[Corollary~4.2]{geze12},
but it need not be transverse to~$S$.)
By adding a half-symplectisation along $B$ and taking
a graph in that half-symplectisation --- analogous
to what we said at the beginning of this section ---
we obtain a cobordism (which we continue to call $A_{BS}$)
with an exact symplectic form $d\mu$ such that $\mu =\rme^t\alpha_B$
near $B$, up to constant scale, with $t\in [0,\varepsilon)$, say.
This $d\mu$ is the symplectic form coming from the Stein
structure away from the collar we added to~$B$, but $\mu$ need
not induce a contact structure on~$S$.

In \cite{geze12} we constructed a (compact, but non-closed!) symplectic
cap $(C_{\infty},\omega_{\infty})$ for a hypersurface $S$ of the
type described above. We adapt this construction to the present setting.
Here we shall even work with a non-compact cap. Since
the behaviour of holomorphic spheres near the boundary resp.\ in the
non-compact end is well controlled, this issue is inessential.

Thus, we form the symplectic manifold $(C_{\infty},\omega_{\infty})$ by
compactifying the $\C$-factor in $\psi^{-1}([c,\infty))$ to
a $\CP^1$ with a scaled Fubini--Study form. Note that $C_{\infty}$
contains the complex hypersurface $V_{\infty}:=V\times\{\infty\}$.
We then glue this with the described cobordism to obtain the
symplectic manifold
\[ (C,\omega_C):=(A_{BS},d\mu)\cup_S(C_{\infty},\omega_{\infty})\]
with $(B,\alpha_B)$ as the strong concave boundary. This places
us in the setting of Theorem~\ref{thm:cobordism}.

Moreover, the primitive $\lambda$ of $\omega_0$ on $(-\infty,0]\times
B\times D^2$, introduced in the proof of Theorem~\ref{thm:cobordism},
extends to a primitive on $A_{BS}\times D^2$, which will be
relevant for the subsequent analysis of holomorphic curves
(specifically, the proof of Lemma~\ref{lem:plane}).

Now construct $(W,\omega)$ as in Theorem~\ref{thm:cobordism}, and complete
this symplectic manifold at the negative end, i.e.\ form the
manifold
\[ \wtW:=(-\infty,-2]\times M\cup_M W\]
with symplectic form
\[ \tilde{\omega}:=\begin{cases}
                   d(\rme^s\alpha) & \text{on $(-\infty,-2]\times M$},\\
                   \omega          & \text{on $W$}.
                  \end{cases}
\]

\subsection{The almost complex structure on~$\wtW$}
On the symplectic manifold $(\wtW,\tilde{\omega})$ we choose an almost
complex structure $J$ tamed by~$\tilde{\omega}$. Outside the
collar $[-\varepsilon/2,0]\times N$ and $C_{\infty}\times D^2$, the
almost complex structure may be chosen generically. Inside these
regions, we impose the following conditions. Recall that we write
$R$ for the Reeb vector field of the stable Hamiltonian structure
$(\omega_N,d\theta)$ on the collar $(-\varepsilon,0]\times N$,
i.e.\ $R\ip\omega_N=0$ and $d\theta(R)=1$, see Lemma~\ref{lem:collar}.
Also recall that on $C\times S^1\subset N$ the $S^1$-fibration
of $N$ is the obvious one, and the restriction of $\omega_N$
to the fibres $C$ equals~$\omega_C$.

By $J_{\infty}$ we denote the complex structure on $C_{\infty}$
coming from the inclusion $C_{\infty}\subset V\times\CP^1$.

\begin{itemize}
\item[(J1)] Over $C_{\infty}\times D^2$ we take $J$ to be a split structure
$J_{\infty}\oplus j$, where $j$ is given by
$j(\partial_{\rho})=\partial_{\theta}$
for $\rho\in [1-\varepsilon/2,1]$. (Here we rely on condition~(f-iii).)
\item[(J2)] Over $[-\varepsilon/2,0]\times N$ we take
$J$ to respect the splitting $\ker d\theta\oplus
\bigl\langle\partial_{\rho},R\bigr\rangle$ (so that in particular
the fibres of each $\{\rho\}\times N$ are holomorphic);
on $\bigl\langle\partial_{\rho},R\bigr\rangle$ we require
$J(\partial_{\rho})=R$; on $\ker d\theta$ we choose $J$
compatible with $\omega_N|_{\ker d\theta}$.
\item[(J3)] On the cylindrical end $(-\infty,0]\times M$,
the almost complex structure $J$ is cylindrical
and symmetric in the sense of \cite{behwz03}, i.e.\ it
preserves $\xi=\ker\alpha$ and satisfies $J\partial_s=R_{\alpha}$.
\end{itemize}

Note that on the overlap of the two regions specified in (J1)
and (J2) we have $R=\partial_{\theta}$, so an almost complex structure
satisfying (J1) automatically satisfies (J2). By Lemma~\ref{lem:collar},
an almost complex structure satisfying (J2) is compatible with
$\omega$ on the collar $[-\varepsilon/2,0]\times N$.
\subsection{Holomorphic spheres in $(\wtW,J)$}
We now recall from \cite{geze12} the essential information about
holomorphic spheres in~$C_{\infty}$. As in that paper, we consider a
closed neighbourhood $U_{\partial}$ of the positive end of~$\wtW$,
defined by
\[ U_{\partial}:=\{ x\in V\co\psi_V(x)\geq c\}\times
\CP^1\times D^2\subset C_{\infty}\times D^2\subset\wtW,\]
where $c$ is the level used in the definition of $S=\psi^{-1}(c)$.
Observe that $U_{\partial}$ is foliated by holomorphic spheres
$\{ p\}\times \CP^1\times\{ w\}$.

\begin{lem}
\label{lem:spheres}
Let $u\co \CP^1\rightarrow\wtW$ be a smooth non-constant $J$-holomorphic
sphere.
\begin{itemize}
\item[(i)] If $u(\CP^1)\cap (C_{\infty}\times D^2)\neq\emptyset$,
then $u(\CP^1)\cap(V_{\infty}\times D^2)\neq\emptyset$.
\item[(ii)] If $u(\CP^1)\cap U_{\partial}\neq\emptyset$, then
$u(\CP^1)\subset U_{\partial}$, and $u$ is of the form $z\mapsto
(p,v(z),w)$, where $v$ is some holomorphic covering of $\CP^1$ by
itself.
\item[(iii)] If $u(\CP^1)\subset C_{\infty}\times D^2$, then $u$
is one of the spheres in~{\rm (ii)}.
\end{itemize}
\end{lem}

The proof is completely analogous to that of~\cite[Lemma~5.2]{geze12}.
\subsection{The moduli space of holomorphic spheres}
Fix a holomorphic sphere
\[ F:=\{p_0\}\times\CP^1\times\{w_0\}\subset U_{\partial},\]
where the choice is made such that $F$ is not contained in the collar
neighbourhood $[-\varepsilon/2,0]\times N$. Write $\wtMM$ for the moduli
space of smooth $J$-holomorphic spheres $u\co\CP^1\rightarrow\wtW$
representing the class $[F]\in H_2(\wtW)$ and not contained
in $[-\varepsilon/2,0]\times N$. This last condition we impose
to avoid problems with transversality caused by the non-generic
almost complex structure on the collar. The homological
intersection number of $[F]$ with the class in $H_{2n}(\wtW,U_{\partial}\cup
C_{\infty}\times\partial D^2)$ represented by the complex hypersurface
$V_{\infty}\times D^2$ equals~$1$.

In conjunction with Lemma~\ref{lem:spheres} we see that any holomorphic
sphere that touches $U_{\partial}$ or is contained in $C_{\infty}\times D^2$
is of the form $z\mapsto (p,v(z),w)$ for some $v\in\Aut(\CP^1)$.
This tells us that all holomorphic spheres corresponding to
points in $\wtMM$ near its ends are standard.

As in \cite[Proposition~6.1]{geze12}, we have the following statement.

\begin{prop}
The moduli space $\wtMM$ is a smooth manifold of dimension $2n+6$,
and all elements of $\wtMM$ are simple holomorphic spheres.
\end{prop}

The spheres in $\wtMM$ being simple implies that the
quotient space
\[ \MM:=\wtMM\times_{\Aut(\CP^1)} \CP^1\]
is a smooth manifold of dimension $2n+2$.
\subsection{Spheres intersecting an arc}
Let $\gamma$ be a proper embedding of the interval $[0,1)$
into $\wtW$ with $\gamma (0)\in F$ and $\gamma(t)\in
(-\infty,-2]\times M$ for $t$ near~$1$. In addition, we require that
the image of $\gamma$ lie in the complement of the collar
neighbourhood $[-\varepsilon/2,0]\times N$. This ensures that the
space
\[ \MM_{\gamma}:=\ev^{-1}(\gamma),\]
where $\ev$ is the evaluation map
\[ \begin{array}{rccc}
\ev\co & \MM   & \lra        & \wtW\\
       & [u,z] & \longmapsto &  u(z),
\end{array}\]
is a smooth $1$-dimensional manifold (with boundary).

Any sphere in $\MM$ represents the class $[F]$, which has
homological intersection number zero with any fibre of~$N$. Thus,
by positivity of intersection~\cite[Proposition~7.1]{cimo07},
any sphere in $\MM$ is either disjoint from the collar
$[-\varepsilon/2,0]\times N$ or contained in a single fibre
of some $\{\rho\}\times N$. Therefore, our choice of $\gamma$
ensures that no sphere in $\MM_{\gamma}$ can touch the collar.

Near the preimage of $\gamma(0)$, the manifold $\MM_{\gamma}$ looks
like a half-open interval, but there are no other boundary points.
In particular, $\MM_{\gamma}$ is non-compact.
\subsection{The conclusion of the argument}
By the non-compactness of $\MM_{\gamma}$ we find a sequence of
points in $\MM_{\gamma}$ without any convergent subsequence.
The compactness theorem from symplectic field
theory~\cite[Theorem~10.2]{behwz03}, however, implies the existence
of a subsequence convergent in the Gromov-Hofer sense to
a holomorphic building of height $k_-|1$. We want to show
that $k_->0$, which is equivalent to saying that the subsequence
is \emph{not} Gromov-convergent to a bubble tree in~$\wtW$.

By construction, the symplectic form $\tilde{\omega}$ is exact on
$\wtW\setminus(C_{\infty}\times D^2)$. Hence, any non-constant sphere
in a purported bubble tree must intersect $C_{\infty}\times D^2$,
and further, by Lemma~\ref{lem:spheres}~(i), each sphere
intersects the complex hypersurface $V_{\infty}\times D^2$.
By positivity of intersection~\cite{cimo07} and the homological
intersection of $F$ with $V_{\infty}\times D^2$ being~$1$, there
can be no non-trivial bubble trees. So, indeed, a non-convergent
sequence in $\MM_{\gamma}$ has to break into a building with
$k_->0$.

The proof of Theorem~\ref{thm:main} is completed with the
following lemma.

\begin{lem}
\label{lem:plane}
The described building contains a finite energy plane with
a positive puncture, corresponding to a contractible
Reeb orbit.
\end{lem}

\begin{proof}
In any holomorphic building of height $k_-|1$ with $k_->0$
there must be at least two finite energy planes. By our intersection
argument, at most one of these can intersect $V_{\infty}\times D^2$.
Any other finite energy plane in the top level $W$
of the building would have to stay outside $C_{\infty}\times D^2$. This,
however, is impossible, because there the symplectic
form $\tilde{\omega}$ is exact (see the comment
in Section~\ref{section:cobordism},
right before the construction of~$\wtW$), so by Stokes's theorem
the Hofer energy of a finite energy plane with a negative end
would be negative, cf.~\cite[Lemma~5.16]{behwz03}.
Likewise, there can be no finite energy plane with a negative
puncture in any of the lower levels of the building,
all of which are copies of the symplectisation of $M$
(and hence exact). We conclude that there must be a
finite energy plane with a
positive puncture in one of the lower levels
$(\R\times M,d(\rme^s\alpha))$ of the building.
\end{proof}
\section{Examples}
In this section we discuss some applications and mild extensions
of our main theorem. The following list is meant to be illustrative rather
than exhaustive.

\vspace{1mm}

(1) The contrapositive of our theorem allows us to draw topological
conclusions. Given a contact manifold $(M,\xi=\ker\alpha)$ with
$R_{\alpha}$ not having any contractible periodic Reeb orbit,
one can deduce that $\xi$ is not supported by an open book
with subcritical pages. Examples are provided by the unit
cotangent bundle of any $n$-torus $\R^n/\Lambda$ obtained as a
quotient of $\R^n$ with the euclidean metric, since the
Reeb flow on the unit cotangent bundle of any Riemannian manifold
coincides with the cogeodesic flow, cf.~\cite[Theorem~1.5.2]{geig08}.

\vspace{1mm}

(2) In our construction, $A_{BS}$ was a symplectic cobordism with
an exact symplectic form. Under the weaker assumption that $A_{BS}$
is any semi-positive symplectic cobordism, one can still infer the
existence of at least a nullhomologous Reeb link as follows.

As in the proof of Theorem~\ref{thm:main}, we need only
exclude Gromov-convergent sequences. Any potential bubble tree
arising as a Gromov limit of a sequence in $\MM_{\gamma}$ has to
stay outside the collar $[-\varepsilon/2,0]\times N$ by an intersection
argument. Indeed, at least one sphere in the bubble intersects $\gamma$
and hence is not contained entirely in the collar. If the bubble
tree were to intersect the collar, there would be at least one
sphere having positive intersection with a fibre of (a copy of)~$N$.
On the other hand, we can obviously find a different fibre
(in a different copy of~$N$) having empty intersection with the
bubble tree.

Provided we can show that the cobordism $\wtW\setminus (C_{\infty}
\times D^2)$ is semi-positive, we may then appeal to
\cite[Section~6.3]{geze12} to rule out non-trivial bubble trees.

For $n=2$, semi-positivity of this $6$-dimensional cobordism is
automatic. For $n\geq 3$, we argue as follows. Recall that
\[ \wtW\setminus (C_{\infty}\times D^2)=(-\infty,0]\times M
\cup_{B\times D^2}(A_{BS}\times D^2),\]
which is homotopy equivalent to $M\cup_{B\times D^2}(A_{BS}\times D^2)$.
We have
\[ H_2(M\setminus B\times \Int(D^2),B\times S^1)\cong
H^{2n-1}(M\setminus B\times \Int(D^2))=0\;\;\text{for $n\geq 3$},\]
since $M\setminus B$ fibres over $S^1$ with fibres having the homotopy type
of an $n$-dimensional complex. This implies that any class in
$H_2(\wtW\setminus (C_{\infty}\times D^2))$ is homologous to
one in $A_{BS}\times D^2\simeq A_{BS}$. Moreover, $A_{BS}$ is
a complex hypersurface in $A_{BS}\times D^2$, so the values
of $[\tilde{\omega}]$ and $c_1(\wtW,J)$, respectively, on any
homology class in $H_2(\wtW\setminus (C_{\infty}\times D^2))$
can be computed inside $A_{BS}$, where semi-positivity holds.

\vspace{1mm}

(3) A Lagrangian embedding of the $2$-sphere into $\C^2$ blown up in
two points can be defined as follows, cf.~\cite[Example~2.14]{seid08}.
Start with the Lagrangian cylinder
\[ \{ (r,\rme^{\rmi\varphi})\in\C^2\co r\in [0,3],\; \varphi\in[0,2\pi)\}.\]
Now form a symplectic blow-up by cutting out
the open unit ball in $\C^2$ centred at $(0,0)$ and collapsing
the fibres of the Hopf fibration on the boundary sphere.
We maintain that the cylinder descends to a Lagrangian disc
in the blown-up manifold. A similar blow-up centred at $(3,0)$
then produces a Lagrangian $2$-sphere in $\C^2\#\overline{\CP}^2
\#\overline{\CP}^2$.

The most convenient way to see that the cylinder descends to a smooth disc
in the blown-up manifold is Lerman's description of a blow-up as a
symplectic cut~\cite{lerm95}. That is, on $\C^2\times \C$ with coordinates
$(z_1,z_2,w)$ and standard symplectic form we consider the
$S^1$-action given by
\[ \rme^{\rmi\theta}(z_1,z_2,w)=(\rme^{\rmi\theta}z_1,\rme^{\rmi\theta}z_2,
\rme^{-\rmi\theta}w).\]
The function $\psi\co\C^2\times\C\rightarrow\R$ given by
$\psi(z_1,z_2,w)=|z_1|^2+|z_2|^2-|w|^2$ is $S^1$-invariant,
and so is the symplectic form.
The blown-up manifold is the smooth symplectic quotient of the level set
$\psi^{-1}(1)$ under the free $S^1$-action.

The isotropic cylinder 
\[ \{ (r,\rme^{\rmi\varphi},r)\in\C^2\times\C
\co r\in [0,3],\; \varphi\in[0,2\pi)\}\subset\psi^{-1}(1) \]
descends to a disc in the blown-up manifold which is Lagrangian
outside its centre. Provided the disc is smooth at the centre,
it will be Lagrangian by continuity. We claim that
\[ (r,\varphi)\longmapsto [r:\rme^{\rmi\varphi}:r]\]
defines a smooth parametrisation of the disc.
Beware that the `homogeneous' coordinates here refer to the quotient
under the $S^1$-action only. Indeed, we have
\[ [r:\rme^{\rmi\varphi}:r]=[r\rme^{-\rmi\varphi}:1:r\rme^{\rmi\varphi}],\]
so the parametrisation can be rewritten as $z\mapsto [\oz:1:z]$.

Now, by Weinstein's neighbourhood theorem for Lagrangian
submanifolds~\cite{wein71},
a small cotangent disc bundle of $S^2$ embeds symplectically into
$\C^2\#\overline{\CP}^2\#\overline{\CP}^2$. This gives us
a symplectic cobordism from $\RP^3$ to $S^3$ (with their standard contact
structures), which is semi-positive for dimensional reasons.

Then Example~(2) applies to show that the contact manifold
obtained from an open book whose page is the cotangent disc bundle of $S^2$,
and whose monodromy is a $k$-fold Dehn twist $\tau^k$, $k\in\Z$,
carries a nullhomologous Reeb link. For $k\geq 0$ the resulting manifolds
are the $5$-dimensional Brieskorn manifolds $\Sigma(k,2,2,2)$
with their natural contact structure,
see~\cite{koni05}; for $k<0$ we obtain a contact structure
on a diffeomorphic copy of $\Sigma(|k|,2,2,2)$.
Since these manifolds are simply connected, see~\cite{hima68},
we obviously get a contractible periodic Reeb orbit.
The reader may wish to compare this with the results
of van Koert~\cite{koer08}.

It was shown by Seidel~\cite[Proposition~2.4]{seid08} that
the Dehn twists $\tau^k$ for different $k\in\Z$ are not
symplectically isotopic (with compact support), although
$\tau^2$ is topologically isotopic to the identity.

\vspace{1mm}

(4) A Lagrangian embedding of the $3$-sphere into $\CP^1\times\C^2$
can be defined by
\[ \C^2\supset S^3\ni(z_1,z_2)\longmapsto\bigl([z_1:z_2],(\oz_1,\oz_2)\bigr)
\in\CP^1\times\C^2,\]
see~\cite[Example~2.2.8]{alp94}.
Analogous to Example~(3), and with $DT^*S^3$ denoting a small
cotangent disc bundle of~$S^3$ and $B\subset\C^2$ a large ball,
we can build the symplectic cap
\[ C:=\bigl((\CP^1\times B)\setminus DT^*S^3\bigr)\cup
\bigl((\CP^1\times\C^2)\setminus(\CP^1\times B)\bigr).\]
Over the first part of $C$ (crossed with~$D^2$)
we allow ourselves a generic choice of
almost complex structure, on the second part we choose the obvious
complex structure, which gives us the holomorphic spheres in
the cap.

The intersection argument that we used in the proof of our main theorem
in order to show that the spheres in our moduli space are simple
no longer applies, since the hypersurface $\{\infty\}\times\C^2$
is not contained in~$C$. Instead however, we may reason
as in Example~(2) that a class $H_2(\wtW)$ can
be regarded as a class in $H_2(C)$. From
$\tilde{\omega}([F])=\omega_{\FS}([F])=\pi$ (with $F=\CP^1\times *$),
and the observation that $\tilde{\omega}$ takes values in $\pi\Z$ on
spherical classes, we conclude that spheres in the class $[F]$ are simple.
Also as in Example~(2) we see that semi-positivity (which is automatic on
$C$ for dimensional reasons) carries over to the relevant $8$-dimensional
cobordism.

Thus, as in Example~(3) we see that the open book with page the cotangent
disc bundle of $S^3$ and monodromy a $k$-fold Dehn twist $\tau^k$
supports a contact structure with a contractible periodic Reeb orbit.
The resulting manifolds, analogous to~(3), are the Brieskorn manifolds
$\Sigma (|k|,2,2,2,2)$. These have homology $H_3\cong\Z_{|k|}$,
cf.~\cite[p.~37]{koer05}. Hence,
in this dimension the Dehn twists $\tau^k$ for different $k\in\Z$
are not even topologically isotopic.
\begin{ack}
We are grateful to Sam Lisi for questions and suggestions that
prompted the present paper.
We thank Emmanuel Giroux for indicating to us how to prove
Proposition~\ref{prop:normal}, and Fan Ding for help with
a technical point in that proof.
We thank Janko Latschev for
useful discussions. Some parts of this research 
were done during the workshop
`From Conservative Dynamics to Symplectic and Contact Topology'
at the Lorentz Center, Universiteit Leiden, in August 2012.
We thank the Lorentz Center and its efficient staff for providing an
excellent research environment.
\end{ack}

\end{document}